\documentclass[leqno,12pt]{article}
\usepackage{inputenc}
\usepackage[english]{babel}
\usepackage{amsmath}
\usepackage{amssymb}
\usepackage{amsthm}
\usepackage{xcolor}
\usepackage[all]{xy}
\usepackage{multicol}
\usepackage{tikz}
\usetikzlibrary{backgrounds}
\date{}
\newtheorem{thm}{Theorem}[section]
\newtheorem{lem}[thm]{Lemma}
\newtheorem{prop}[thm]{Proposition}
\newtheorem{cor}[thm]{Corollary}

\theoremstyle{definition}
\newtheorem{ex}[thm]{\it Example}

\newtheorem{defin}[thm]{Definition}
\numberwithin{equation}{section}
\usepackage{verbatim}

\usepackage{graphicx}
\usepackage{newlfont}

\usepackage{mathtools}

\usepackage{latexsym}

\newcommand{\N}{\mathbb{N}}
\newcommand{\Z}{\mathbb{Z}}

\newcommand{\cont}{\subseteq}

\newcommand{\nin}{\not \in}

\newcommand{\K}{K}
\newcommand{\g}{\gamma}
\newcommand{\D}{\Delta}
\newcommand{\al}{\boldsymbol{\alpha}}
\newcommand{\be}{\boldsymbol{\beta}}
\newcommand{\ga}{\boldsymbol{\gamma}}
\newcommand{\de}{\boldsymbol{\delta}}

\newcommand{\bff}{\textbf}
\newcommand{\newbrace}[1][]{
\begin{tikzpicture}[baseline=-0.5ex]
\draw[#1] (0,0) -- (0.3,0.3);
\draw[#1] (0,0) -- (0.3,-0.3);
\end{tikzpicture}
}

{\;\newbrace[#1]\;\begin{array}{@{}l@{}}}%
{\end{array}}

\everymath{\displaystyle}

\usepackage{hyperref}

\title{Almost symmetric good semigroups}

\author{L. Casabella\thanks{Dipartimento di Matematica "Tullio Levi-Civita", Universit\`a degli Studi di Padova, Italy; {\em email}: laura.casabella1@gmail.com} \and
M. D'Anna\thanks{Dipartimento di Matematica e Informatica, Universit\`a degli Studi di Catania, Italy; {\em email}: mdanna@dmi.unict.it - ORCID: 0000-0002-9943-5527}}
\begin{document}

\maketitle

\begin{abstract}
The class of good semigroups is a class of subsemigroups of $\mathbb N^h$, that includes 
the value semigroups of rings associated to curve singularities and their blowups, and 
allows to study combinatorically the properties of these rings.
In this paper we give a characterization of almost symmetric good subsemigroups of $\mathbb N^h$,
extending known results in numerical semigroup theory and in one-dimensional ring theory,
and we apply these results to obtain new results on almost Gorenstein one-dimensional
analytically unramified rings.


\end{abstract}

\noindent 
{\bf Acknowledgments.} The second author was partially funded by the project “Propriet\`a locali e globali
di anelli e di variet\`a algebriche”-PIACERI 2020–22, Universit\`a degli Studi
di Catania.

\section*{Introduction}


Let $\mathbb N$ be the set of nonnegative integers; \emph{good semigroups} are a class of subsemigroups of $\N^h$, with $h \in \mathbb N$, $h \geq 1$ (see the formal definition in Section 1). Their definition was given in \cite{a-u} in order to describe value semigroups of noetherian, analytically unramified, 
one-dimensional, semilocal reduced rings, even though the properties of these value semigroups were already studied in \cite{two}, \cite{c-d-gz}, \cite{C-D-K}, \cite{danna1}, \cite{felix}, \cite{delgado}, \cite{garcia}. The class of rings mentioned above includes the local rings arising from curve singularities (and from their blowups), possibly with more than one branch. However, as shown in \cite{a-u}, there are good semigroups that are not value semigroups of any ring. 


\medskip
If $S$ is a value subsemigroup of $\mathbb N^h$, the integer $h$ represents the number of branches of the singularity. In the one branch case, the value semigroup is a numerical semigroup, and the algebraic properties of the ring can be tran\-slated and studied at semigroup level. Some of the results in this direction have been proved to hold for $h>1$ as well. For instance, the numerical cha\-racterization of a one-dimensional analytically irreducible local Gorenstein domain via symmetric numerical semigroups (see \cite{kunz}) can be generalized to analytically unramified local rings (see \cite{delgado} and also \cite{C-D-K}); analogously, the numerical characterization of the canonical module in the analytically irreducible case (see \cite{jager}) can be given also in the more general case (see \cite{danna1}).

Dealing with birational extensions of an analytically unramified ring $R$, one has to work with both local and semilocal rings; this fact is well represented at semigroup level, since it is possible to give a satisfying definition of local good semigroup that, for value semigroups, corresponds to the local ring case; with this definition one can show that any good semigroup is a direct product of local good semigroups that, when $S=v(R)$, correspond to the value semigroups of the localizations of $R$ at its maximal ideals (see \cite{a-u}).

\medskip
These facts motivated the interest in studying good semigroups, and, since not every good semigroup is a value semigroup of some ring, it is often necessary to work on good semigroups in a purely combinatorial setting.
Despite their name, good semigroups present some problems that make their study difficult; e.g. they are not finitely generated as monoids (but, on the other hand, they can be completely determined by a finite set of elements; see \cite{garcia} and \cite{D-G-M-T}) and they are not closed under finite intersections. 
Moreover, if we consider good ideals of good semigroups (e.g. the ideals
arising as values of ideals of the corresponding ring), this class is not closed under sums and differences (see, for instance, \cite{a-u} and \cite{KST}).
Therefore, unlike what happens for numerical semigroups (in analogy to analytically irreducible domains), it is not easy to define for good semigroups some concepts that translates analogous ring concepts, like the embedding dimension and the type (see \cite{DGM2} and \cite{MZ}).
Finally, when working with good semigroups, technical problems often arise if $h\geq 3$ and, for this reason,
several papers on the subject are limited to the very special case $h=2$ (see e.g. \cite{garcia}, \cite{two}, \cite{DGM}
and \cite{DGM2}).

\medskip
Almost Gorenstein analytically unramified rings and almost symmetric numerical semigroups were introduced in \cite{BF} in order to find two classes of rings and numerical semigroups, respectively, that are close to Gorenstein rings and to symmetric numerical semigroups, in certain aspects.
The notion of almost Gorenstein ring was generalized for any one-dimensional ring in \cite{GMP} and, subsequently, for rings in any dimension admitting a canonical ideal (see \cite{GTT}). Since then, almost Gorenstein rings have been intensively studied, being a useful intermediate class between Cohen-Macaulay and Gorenstein rings. Similarly, the notion of almost symmetric numerical semigroups was generalized to good semigroups in \cite{a-u}.

\medskip
A remarkable property shown in \cite{BF} for a one-dimensional analytically unramified local ring $(R, \mathfrak m)$ is that such a ring is almost Gorenstein of maximal embedding dimension if and only if $\mathfrak m:\mathfrak m$ is Gorenstein, and an analogous result was proved in the same paper for numerical semigroups. These two results were the starting point for other similar investigations. In particular, in \cite{GMP} the authors generalized the first one for any one-dimensional Cohen-Macaulay local ring admitting a canonical ideal and, in \cite{DS2}, the following more general result is proved:

\vspace{5pt}
\noindent {\bf Proposition} \emph{\cite[Proposition 3.2]{DS2}. Let $(R, \mathfrak m)$ be a one-dimensional Cohen-Macaulay local ring
admitting a canonical ideal. Then
$R$ is a one-dimensional almost Gorenstein
ring if and only if $\mathfrak m$ is a canonical ideal of $\mathfrak m: \mathfrak m$.}

\vspace{5pt}
Note that an analogous result was stated in a non-noetherian setting in \cite{B1}, but with an additional hypothesis. 
As for numerical semigroups, the corresponding result of the above proposition was proved in \cite{B}.

\medskip
If we consider good semigroups, fewer results of this kind are known. More precisely, a partial result was proved in \cite{DGM2} for the case $h=2$, assuming that $S$ is a local good semigroup of maximal embedding dimension. In this context, the main goal of this paper is to prove the following result:

\vspace{5pt}
\noindent
{\bf Theorem 2.4.} \label{main} \emph{Let $h \geq 2$, let
$S \subseteq \mathbb N^h$  be a local almost symmetric good semigroup and let $M=S\setminus \{\boldsymbol{0}\}$
be its maximal ideal.
Then $S$ is almost symmetric if and only if $M-M=\{\al \in \mathbb N^h \mid \al+M \subseteq M\}$ is a good semigroup and $M$ is a canonical ideal of $M-M$.}

\vspace{5pt}
Hence, after recalling in Section 1 the basic notions necessary to develop our arguments, 
in Section 2 we give a description of the canonical ideal of $M-M$, assuming it is a good semigroup (Proposition \ref{miracolo}), and we use it to prove the main theorem,
together with a characterization of local almost symmetric good semigroups of maximal embedding dimension via $M-M$ (Corollary \ref{alm-symm-good}). Afterwards, we generalize these results to
the non-local case and, in Section 3, we show how our results imply new characterizations of almost Gorenstein analytically unramified one-dimensional local or semilocal rings (see, e.g. Proposition \ref{bella}).

\section{Preliminaries on good semigroups}

Let $\mathbb N$ be the set of nonnegative integers and let $h\geq 2$ be an integer. As usual, $\le$ stands for the natural partial ordering on $\mathbb N^h$: $\al \le \boldsymbol{\beta}$ if $\alpha_i\le \beta_i$ for every $i=1, \dots, h$.
Given $\boldsymbol{\alpha},\boldsymbol{\beta}\in \mathbb N^h$, the infimum of the set $\{\boldsymbol{\alpha},\boldsymbol{\beta}\}$ (with respect to $\le$) will be denoted by $\boldsymbol{\alpha}\wedge \boldsymbol{\beta}$. Hence $\boldsymbol{\alpha}\wedge \boldsymbol{\beta}=(\min(\alpha_1,\beta_1),\dots, \min(\alpha_h,\beta_h))$.

A submonoid $S \cont \N^h$ is said to be a \emph{good semigroup} if the following
three properties are satisfied:

\begin{itemize}

    \item[\bff{(G1)}] if $\al, \be \in S$, then $\al \wedge \be \in S$;
    \item[\bff{(G2)}] if $\al, \be \in S$, $\al \neq \be$ and $\alpha_i=\beta_i$ for some $i \in \{1, \dots, h \}$, then there exists $\de \in S$ such that $\delta_i > \alpha_i=\beta_i$ and $\delta_j \geq \min(\alpha_j, \beta_j)$ for all $j \neq i$ (and equality holds if $\alpha_j \neq \beta_j$);
    \item[\bff{(G3)}] there exists $\boldsymbol{c} \in \N^h$ such that $\boldsymbol{c} + \N^h \cont S$.
    
\end{itemize}

Notice that for $h=1$, properties \bff{(G1)} and \bff{(G2)} become meaningless
and property \bff{(G3)}, for a submonoid $S$ of $\mathbb N$, means that $S$ is
a numerical semigroup. Hence we can consider numerical semigroups as the good semigroups in the case $h=1$.

A \emph{relative ideal} of a good semigroup $S \subseteq \N^h$ is a subset $E$ of $\mathbb Z^h$ such that $\al+\be \in E$ for any $\al \in S$ and $\be \in E$, and there is a $\de \in S$, such that $\de +E \subseteq S$. If $E\subseteq S$ then we will simply say that $E$ is an \emph{ideal} of $S$.
A relative ideal $E$ satisfying \bff{(G1)} and \bff{(G2)} is called \emph{good ideal}. Note that a relative ideal necessarily satisfies \bff{(G3)}, since it is closed up to sums with elements of $S$.
Observe that, given two relative ideals $E, F$ of $S$, the following sets are relative ideals as well: 
\begin{itemize}
    \item $E+F := \{\al+\be \mid \al \in E, \be \in F \}$,
    \item $E-F := \{ \al \in \Z^d \mid \al+F \cont E \}$.
\end{itemize}
However, if $E$ and $F$ are good, it is not true in general that $E+F$ and $E-F$ are good, as shown in \cite[Example 2.10]{a-u}.

\medskip
If the only element of $S$ with a zero component is $\boldsymbol{0}=(0,0, \dots 0)$, we will say that $S$ is \emph{local}. In this case, the set $M:=S \setminus \{\boldsymbol{0}\}$ is a
good ideal and we will call it the \emph{maximal ideal} of $S$. Moreover, by property \bff{(G1)}, a local good semigroup has a minimum nonzero element. We will denote it by $\boldsymbol{e}=(e_1,\dots, e_h)$ and call it the \emph{multiplicity} vector of $S$.

\medskip
In \cite{a-u} it is proved that any good semigroup can be uniquely expressed as a direct product of local ones, i.e. $S=S_1\times \dots \times S_r$ and, denoting by $M_i$ the maximal ideal of $S_i$, we will denote by $J$ the product $M_1\times \dots \times M_r$. In this case, $J$ is a good ideal, that we will call the \emph{Jacobson ideal} of $S$. Notice that any $S_i$ in this expression is a
subsemigroup of $\mathbb N^{h_i}$, with $h_1+h_2+\dots+h_r=h$, and the subset of components $\{j_1, \dots ,j_{h_i}\}\subseteq \{1,\dots,h\}$, where $S_i$ lives,
will be called the \emph{support} of $S_i$.

The above notations will be fixed for the rest of the paper. In addition, we will always assume that $S \neq \N^h$.

Notice that, for any relative ideal $E$,
it holds true that $E-E$ is a submonoid of $\N^h$, not necessarily good, containing $S$. 
In particular, if $S$ is a local good semigroup, then
$S-M=M-M=\{\al \in \mathbb N^h \mid \al+M \subseteq M\}$ is both a monoid 
such that $S \subsetneq M-M \subseteq \N^h$ and a relative ideal of $S$.
The same holds, in general, for $J-J$. 

In the next section, we will need the following lemma in order to deal with the non-local case. Its proof is straightforward. 

\begin{lem}
Let $S\subset \N^h$ be a good semigroup and let $S=S_1\times \dots \times S_r$ be its representation
as direct product of local good semigroups. Then $J-J=(M_1-M_1)\times \dots \times (M_r-M_r)$.
\end{lem}

A relative ideal that plays a central role in the theory of good semigroups is the standard canonical ideal. In order to give its definition, we need to introduce some more notation.

\begin{defin}
Let $\al \in \Z^d$, $i \in \{1, \dots, d \}$. We define:
	        \begin{itemize}
	            \item[] $\D_i(\alpha) := \{ \beta \in \Z^d \mid \alpha_i=\beta_i \text{, } \alpha_j < \beta_j \enspace \forall j \neq i \}$;
	            \item[] $\D(\alpha) := \{ \beta \in \Z^d \mid \exists \, i 
	            \text{ such that } \alpha_i=\beta_i \text{, } \alpha_j < \beta_j \enspace \forall j \neq i \}= \bigcup_i \Delta_i(\alpha)$;
	            \item[] $\Delta_i^S(\alpha) :=\Delta_i(\alpha) \cap S$;
	            \item[] $\Delta^S(\alpha) :=\Delta(\alpha) \cap S$.
	        \end{itemize}
\end{defin}
The minimum $\boldsymbol{c} \in \N^h$ satisfying \bff{(G3)} will be called the \emph{conductor} of $S$. 
We additionally set $\ga := \boldsymbol{c} - \boldsymbol{1}$ (where $\boldsymbol{1}=(1,1,\dots, 1)$) and call it the \emph{Frobenius vector} of $S$.
Observe that, by properties \bff{(G1)} and \bff{(G2)}, $\D^S(\ga)= \varnothing$.
\begin{defin}
The \emph{standard canonical ideal} of $S$ is
$$\K(S):=\{ \al \in \Z^d \mid \Delta^S(\ga - \al) = \varnothing \}. $$
\end{defin}
This relative ideal was firstly defined in the local case in \cite{danna1}. In \cite{a-u}
it was noticed that its definition can be stated in the general case 
and that, if $S=S_1 \times \dots \times S_r$, then $K(S)=K(S_1)\times \dots \times K(S_r)$. It is not difficult to see that $\K(S)$ is a good relative ideal of $S$ such that
$S \cont \K(S) \cont \N^h$.

\medskip
A good relative ideal $\K'$ of $S$ is called \emph{canonical ideal} of $S$ if there exists $x \in \Z^d$ such that $\K'=\K(S)+x$.

The class of canonical ideals can be characterized via a duality property that we are going to state. This property was proved in \cite{KST} for the local case, but it can be stated for any good semigroup, using the representation of $\K(S)$ as a direct product of standard canonical ideals. 


\begin{lem}
\label{Lemma tozzo}\cite[Theorem 5.2.7]{KST}
Given a good relative ideal $\K'$ of $S$, the following conditions are equivalent:
\begin{itemize}
    \item[(i)] $\K'$ is a canonical ideal of $S$;
    \item[(ii)] $\K'-(\K'-I)=I$ for every good relative ideal $I$ of $S$.
\end{itemize}
\end{lem}

In the same paper, the authors proved the following key result, that, again, holds true in the non-local case as well. 
\begin{lem}
\label{ideale buono}\cite[Proposition 5.2.7]{KST}
If $E$ is a good relative ideal of $S$, then $\K(S)-E$ is a good relative ideal of $S$.
\end{lem}

We list below three key properties of canonical ideals that will be useful in the next section. Let $K'$ be a canonical ideal of $S$ and let $E, F$ be two relative ideals. Then:
\begin{itemize}
    \item $E\subseteq F$ if and only if $K'-F \subseteq K'-E$;
    \item $E= F$ if and only if $K'-F = K'-E$;
    \item $K'-K'=S$.
\end{itemize}
These properties can be easily deduced by Lemma \ref{Lemma tozzo}, keeping in mind, as for the third one, that $K'-K'$ is independent of translations of $K'$ and that $K(S)=K(S)-S$. 

\medskip

An important class of good semigroups is the class of  symmetric semigroups, which corresponds to the class of Gorenstein rings. 

\begin{defin}
A good semigroup $S$ is called \emph{symmetric} if, for every $\al \in \N^d$, $\al \in S$ if and only if $\D^S(\ga - \al)= \varnothing$.
\end{defin} 

Notice that, for any good semigroup $S$, $\al \in S$ implies $\D^S(\ga-\al)= \varnothing$, but in general the converse is not true. 

The definition of symmetry is clearly equivalent to the condition $S =\K(S)$; hence a good semigroup $S$ is symmetric if and only if every ideal of $S$ is bidual (or reflexive), i.e. $S-(S-E)=E$. In the non-local case, the symmetry of a good semigroup is also equivalent to the property that every component $S_i$ of $S$ is symmetric. This fact was proved explicitly 
in \cite{DGM2} for the case $h=2$, but the same proof holds in general. 
    
\medskip
Good semigroups are not finitely generated in general. Therefore, it is not possible 
to define the embedding dimension of a local good semigroup just by generalizing the definition given for
local rings or for numerical semigroups. A definition of embedding dimension of a local good semigroup is given in \cite{MZ}, but it is very technical and we will not need to use it
in this paper. In any case, keeping in mind that a one-dimensional noetherian local ring $(R,\mathfrak m)$ is of maximal embedding dimension (i.e. its embedding dimension is the same as its multiplicity) if and only if $\mathfrak m:\mathfrak m = x^{-1}\mathfrak m$ (where $x$ is a minimal reduction of $\mathfrak m$) and the same holds for numerical semigroups,
we can say that a local good semigroup is of \emph{maximal embedding dimension}
if $M-M=M-\boldsymbol{e}$. Notice that this condition implies that 
$M-M$ is a good semigroup (not necessarily local).
This definition is consistent with the one
given in \cite{MZ}.

\section{Almost symmetric good semigroups}
	
\subsection{The local case}

\begin{defin}
A local good semigroup $S$ is called \emph{almost symmetric} if $K(S)+M=M$. 
\end{defin}

Notice that this is equivalent to saying that $K(S)+M \subseteq M$, since the reverse inclusion always holds, because $\boldsymbol{0} \in K(S)$.

Moreover, for any local good semigroup,
it holds true that $M-M \cont K(S) \cup \D(\ga)$: if
    $\al \in (M-M) \setminus \D(\g)$ and $\al \nin K(S)$, then there would exist a nonzero $\be \in \D^S(\ga-\al)$; therefore $\al+\be \in \D^S(\ga)$, contradiction.
    Notice that, in the case of numerical semigroups, the above inclusion is reduced to $M-M \cont K(S) \cup F$, where $F$ is the Frobenius number of $S$.
    
In \cite{a-u} it is shown that almost symmetric local good semigroups can be
characterized 
by the equality of these two sets:
\begin{prop}\label{tipo definito}\cite[Lemma 3.5]{a-u}
A local good semigroup $S$ is almost symmetric if and only if $\K(S) \cup \Delta (\ga) = M - M$.
\end{prop}
    
Proposition \ref{tipo definito} shows that, if $S$ is an almost symmetric local good semigroup, then $M-M$ is a good relative ideal of $S$, because it is the union of a good relative ideal and $\D(\g)$. Hence, in this case, $M-M$ is a (not necessarily local) good semigroup.

\medskip
In order to prove Theorem \ref{main}, we firstly need the following.

\begin{prop}\label{miracolo}
Let $S$ be a local good semigroup, of multiplicity vector $\boldsymbol{e}$, such that $M-M$ is a good semigroup. Then the standard canonical ideal of $M-M$ is $$K(M-M)=(K(S)-(M-M))-\boldsymbol{e}.$$
\end{prop}

\begin{proof}
For the sake of simplicity we set $T=M-M$ and $K=K(S)$. Our aim is to show that $K(T)=(K-T)-\boldsymbol{e}$.

Firstly we prove that $K-T$ is a canonical ideal of $T$. 
In fact,  $K-T$ is not only a good relative ideal of $S$, but also of $T$, because 
if $\al, \de \in T=M-M$ and $\be \in K-T$, we get $(\al+\be)+\de =(\al+\de)+\be \in T+(K-T) \subseteq K$, that is $\al+\be \in K-T$. 
By Lemmas \ref{Lemma tozzo} and \ref{ideale buono}, it is sufficient to show that, for every good relative ideal $E$ of $T$, it holds $(K-T)-((K-T)-E)=E$.
Using the fact that for any three relative ideals $A,B,C$ it holds that
$(A-B)-C=A-(B+C)$, we obtain
$$(K-T)-((K-T)-E)=(K-T)-(K-(T+E)).$$ Since $E$ and $K-E$ are both relative ideals of $T$, then $T+E=E$ and $T+(K-E)=K-E$, hence we get
        \begin{align*} 
& (K-T)-(K-(T+E))=(K-T)-(K-E) \\ 
= & K-(T+(K-E))=K-(K-E).
\end{align*}
Being $E$ a good relative ideal of $T \supset S$, it is also a good relative ideal of $S$, therefore $K-(K-E)=E$, again by Lemma \ref{Lemma tozzo}.

It remains to prove that
$T \cont (K-T)-\boldsymbol{e} \cont \N^d$. As for the first inclusion, this is equivalent to saying that $T+\boldsymbol{e} \subseteq K-T$. But $T+\boldsymbol{e}=(M-M)+\boldsymbol{e}
\subseteq M$ and $M+T=M+(M-M) \subseteq M \subset K$; thus $T+\boldsymbol{e} \subseteq M\subseteq K-T$, as desired.

We prove now the second inclusion, $(K-T)-\boldsymbol{e} \cont \N^d$. First of all, we show that $S - \N^h \cont T+\boldsymbol{e}$. Let $\al \in S- \N^h$; we have to prove that $\al-\boldsymbol{e} \in T$. This is equivalent to proving that for any $\de \in M$, $\al-\boldsymbol{e}+\de \in M$. But $\de - \boldsymbol{e} \geq 0$, i.e. $\de - \boldsymbol{e} \in \N^h$, hence $\al+(\de-\boldsymbol{e}) \in S$. Since $\boldsymbol{0}\notin S- \N^d$, because $S\subsetneq \N^h$, and $\al+(\de-\boldsymbol{e})
\geq \al$, we must have $\al+\de-\boldsymbol{e} \in M$.

Finally notice that $K- \N^h=S- \N^h$: indeed, $\boldsymbol{0} \in K$ and so $K+\N^h=\N^h$; therefore $S-\N^h=(K-K)-\N^h=K-(K+\N^h)=K- \N^h$. Hence $K- \N^h \cont T+\boldsymbol{e}$ and
dualizing with $K$ we obtain $K-(T+\boldsymbol{e}) \cont \N^d$, that is our claim.
\end{proof}

We observe that the fact that $K(S)-(M-M)$ is a canonical ideal of $M-M$
can be deduced by \cite[Corollary 5.2.12]{KST}, where the proof uses an explicit
computation of $K-E$ for any relative ideal $E$ of $S$. We included our proof
of this part in the above proposition, for the sake of completeness.

\medskip
We now use Proposition \ref{miracolo} to prove Theorem \ref{main}
\begin{thm}
\label{main}
A local good semigroup $S$ is almost symmetric if and only if $M-M$ is a good relative ideal of $S$ and $M-\boldsymbol{e}$ is the standard canonical ideal of $M-M$.
\end{thm}

\begin{proof}
For the sake of simplicity we set $K=K(S)$, $T=M-M$ and $K'=K(T)$.
\begin{description}
\item[$(\Leftarrow)$] Since the inclusion $T \cont  \K \cup \Delta (\ga)$ holds for every good semigroup,  it is sufficient to show that $K \cont T$. By assumption $K'=M-\boldsymbol{e}$, and by Proposition \ref{miracolo} $K'=(K-T)-\boldsymbol{e}$. Hence $M=K-T$, that, dualizing with $K$, is equivalent to $K-M=K-(K-T)=T$. But since $K$ is a relative ideal of $S$, we have $K \cont K-M=T$, as desired.
\item[$(\Rightarrow)$] We have already noticed that $S$ almost symmetric implies that $T$ is good.
In order to prove the second assertion we firstly need the follo\-wing:

\smallskip
    \noindent \textit{Claim: every good relative ideal $I$ of $T$ is bidual as an ideal of $S$.}\\
    Let $I$ be a good relative ideal of $T$. We have $\K+I \cont T + I = I$,
    where the first inclusion is true because $S$ is almost symmetric and the second one because $I$ is an ideal of $T$. Conversely, $I \cont \K+I$ because $0 \in \K$. Hence
    \begin{equation}
    \K+I=I.
    \end{equation}
        Now suppose that $I$ is not bidual as ideal of $S$, then 
    $I \subsetneq S-(S-I) \cont \K-(S-I)$. Dualizing with $K$, this is equivalent to 
    $S-I \subsetneq \K-I$. 
    Therefore 
    $$S-I \subsetneq \K-I = \K-(\K+I)=(\K-\K)-I = S-I$$ 
    that is a contradiction. Hence the claim holds true.
    
We are now ready to prove that $M$ is a canonical ideal of $T$.
By Lemma \ref{Lemma tozzo}, it is sufficient to show that $M-(M-I) = I$ for every good ideal $I$ of $T$. Firstly, we show that $M-I=S-I$. If not, there exists an element $x \in S-I$ such that $x \nin M-I$. Then $x+I$ is an ideal of $S$ not contained in $M$, i.e.
$x+I=S$; and this means that $S$ is itself a relative ideal of $T$, a contradiction. Hence $M-I=S-I$. 

Since, by the claim, $I$ is bidual as an ideal of $S$, $$M-(M-I) \cont S-(M-I)=S-(S-I)=I. $$ The other inclusion is always true, hence we get  $M-(M-I) = I$, as desired. 
Finally, it is straightforward to check that $T\subseteq M-\boldsymbol{e}\subseteq \N^h$,
hence $M-\boldsymbol{e}$ is the standard canonical ideal of $T$. 
\end{description} 
\end{proof}

\begin{ex} Let $S$ be the semigroup represented by the black dots in the figure on the left;
then its standard canonical ideal $K(S)$ is obtained by adding to $S$ the white dots in the same figure, and it is not difficult to check that $M+K(S) \subseteq M$, i.e. $S$
is almost symmetric. In the figure on the right, $M-M$ is depicted with black dots and $M-\boldsymbol{e}$ is obtained by adding the white dots to it. As prescribed by the previous theorem, $M-\boldsymbol{e}$ is the standard canonical ideal of $M-M$.

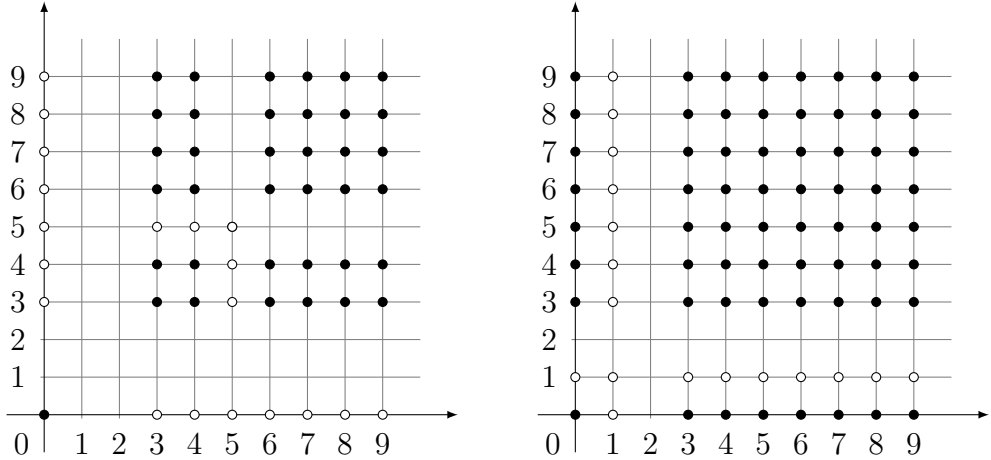
\begin{center}
\begin{tikzpicture}[scale=0.5]
    \coordinate (Origin)   at (0,0);
    \coordinate[label=below:0] (zero) at (-0.6,-0.2);
    \coordinate (XMin) at (-1,0);
    \coordinate (XMax) at (11,0);
    \coordinate (YMin) at (0,-1);
    \coordinate (YMax) at (0,11);
    \draw [-latex] (XMin) -- (XMax);
    \draw [-latex] (YMin) -- (YMax);

    \node[draw,circle,inner sep=1.2pt,fill] at (0,0) {};
    \node[draw,circle,inner sep=1.2pt,fill] at (3,3) {};
    \node[draw,circle,inner sep=1.2pt,fill] at (3,4) {};
    \node[draw,circle,inner sep=1.2pt,fill] at (4,3) {};
    \node[draw,circle,inner sep=1.2pt,fill] at (4,4) {};
    
     \foreach \x in {1,...,9}{
     \coordinate[label=below:\x] (\x) at (\x,-0.2);
         \coordinate[label=left:\x] (\x') at (-0.2,\x);
         \coordinate (start\x) at (\x,-0.1);
         \coordinate ('start\x) at (-0.1,\x);
         \coordinate (end\x) at (\x,10);
         \coordinate ('end\x) at (10,\x);
     }
    
     \foreach \x in {6,...,9}{
      \foreach \y in {6,...,9}{
     \node[draw,circle,inner sep=1.2pt,fill] at (\x,\y) {};
     }
     }
     
     \foreach \x in {3,4,5}{
         \node[draw,circle,inner sep=1.2pt,fill=white] at (\x,5) {};
      \node[draw,circle,inner sep=1.2pt,fill=white] at (5,\x) {};
     }
     
     \foreach \x in {3,4}{
     \foreach \y in {6,...,9}{
     \node[draw,circle,inner sep=1.2pt,fill] at (\x,\y) {};
     \node[draw,circle,inner sep=1.2pt,fill] at (\y,\x) {};
     }
     }
     
    \foreach \x in {3,...,9}{
    \node[draw,circle,inner sep=1.2pt,fill=white] at (\x,0) {};
    \node[draw,circle,inner sep=1.2pt,fill=white] at (0,\x) {};
     }
 
\begin{scope}[on background layer]  
\foreach \x in {1,...,9}{
        \draw [thin, gray] (start\x) -- (end\x);
        \draw [thin, gray] ('start\x) -- ('end\x);}
\end{scope}
\end{tikzpicture} \quad \quad
\begin{tikzpicture}[scale=0.5]
    \coordinate (Origin)   at (0,0);
    \coordinate[label=below:0] (zero) at (-0.6,-0.2);
    \coordinate (XMin) at (-1,0);
    \coordinate (XMax) at (11,0);
    \coordinate (YMin) at (0,-1);
    \coordinate (YMax) at (0,11);
    \draw [-latex] (XMin) -- (XMax);
    \draw [-latex] (YMin) -- (YMax);

    \node[draw,circle,inner sep=1.2pt,fill] at (0,0) {};
    
     \foreach \x in {1,...,9}{
     \coordinate[label=below:\x] (\x) at (\x,-0.2);
         \coordinate[label=left:\x] (\x') at (-0.2,\x);
         \coordinate (start\x) at (\x,-0.1);
         \coordinate ('start\x) at (-0.1,\x);
         \coordinate (end\x) at (\x,10);
         \coordinate ('end\x) at (10,\x);
     }
    
     \foreach \x in {3,...,9}{
     \foreach \y in {3,...,9}{
     \node[draw,circle,inner sep=1.2pt,fill] at (\x,\y) {};
     }
     }
     
     \foreach \x in {1}{
     \foreach \y in {3,...,9}{
     \node[draw,circle,inner sep=1.2pt,fill=white] at (\x,\y) {};
     \node[draw,circle,inner sep=1.2pt,fill=white] at (\y,\x) {};
     \node[draw,circle,inner sep=1.2pt,fill] at (\x-1,\y) {};
     \node[draw,circle,inner sep=1.2pt,fill] at (\y,\x-1) {};
     }
     }

\node[draw,circle,inner sep=1.2pt,fill=white] at (1,1) {};
\node[draw,circle,inner sep=1.2pt,fill=white] at (1,0) {};
\node[draw,circle,inner sep=1.2pt,fill=white] at (0,1) {};
   
\begin{scope}[on background layer]  
\foreach \x in {1,...,9}{
        \draw [thin, gray] (start\x) -- (end\x);
        \draw [thin, gray] ('start\x) -- ('end\x);}
\end{scope}
\end{tikzpicture}
\end{center}
\end{ex}

The following result was proved in \cite{DGM2} for the special case $h=2$. We will give a proof that holds in general, using Theorem \ref{main}.

\begin{cor}\label{alm-symm-good}
A local good semigroup $S$ is almost symmetric of maximal embedding dimension if and only if $M-M$ is a symmetric good semigroup.
\end{cor}

 \begin{proof} As usual, we set $K=K(S)$ and $T=M-M$.
 \begin{description}
 \item[$(\Rightarrow)$] If $S$ is almost symmetric,  by Theorem \ref{main}, $M-\boldsymbol{e}$ is the standard cano\-nical ideal of $T$. Remembering that, by definition, $S$ is of maximal embedding dimension if and only if $T=M-\boldsymbol{e}$, we get that $T$ coincides with its standard canonical ideal, that is equivalent to saying that $T$ is symmetric.
 \item[$(\Leftarrow)$] Assume that $T$ is good and symmetric. Therefore, by Proposition \ref{miracolo}, $T=K(T)=(K-T)-\boldsymbol{e}$.
It follows that $T=K-(T+\boldsymbol{e})\supseteq K-M \supseteq K-S=K$.
Hence $K+M \subseteq M$, which means that is $S$ almost symmetric. 

Now we can apply Theorem \ref{main} to get $\K(T)=M-\boldsymbol{e}$;
therefore, since $T$ is symmetric, $T=K(T)$, i.e. $M-M=M-\boldsymbol{e}$, that means that 
$S$ is of maximal embedding dimension.
\end{description}

\end{proof}

\begin{ex} \label{2.7}
Let $S$ be the semigroup represented by the black dots in the figure on the left; then $M-M$ is the semigroup represented by the black dots in the figure on the right. In this case, $M-M$ coincides with $M-\boldsymbol{e}$ and it is a symmetric good
semigroup. Hence, by the previous corollary, $S$ is almost symmetric of maximal embedding dimension.
The standard canonical ideal $K(S)$ is obtained by adding to $S$ the white dots in the figure on the left, and it is immediately seen that $K(S)\cup \Delta(\boldsymbol{\gamma})=M-M$: this is another way to check that $S$ is almost symmetric.

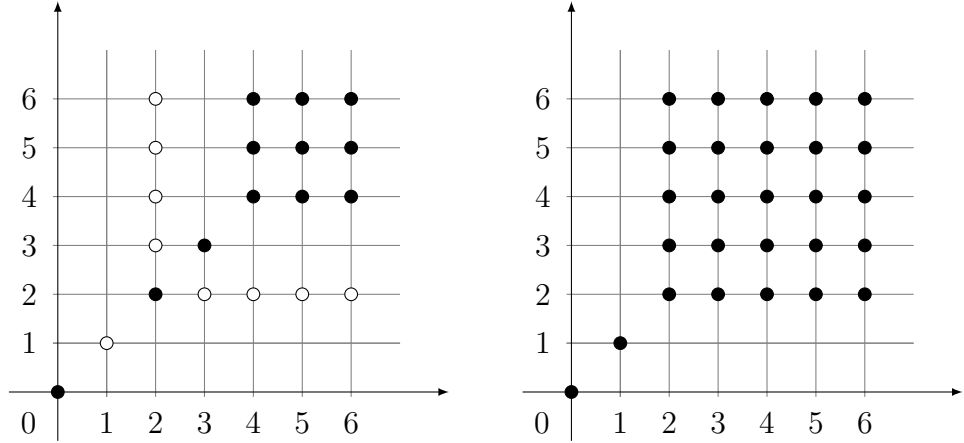
\begin{center}
\begin{tikzpicture}[scale=0.65]
    \coordinate (Origin)   at (0,0);
    \coordinate[label=below:0] (zero) at (-0.6,-0.2);
    \coordinate (XMin) at (-1,0);
    \coordinate (XMax) at (8,0);
    \coordinate (YMin) at (0,-1);
    \coordinate (YMax) at (0,8);
    \draw [-latex] (XMin) -- (XMax);
    \draw [-latex] (YMin) -- (YMax);

    \node[draw,circle,inner sep=1.7pt,fill] at (0,0) {};
    \node[draw,circle,inner sep=1.7pt,fill=white] at (1,1) {};
    \node[draw,circle,inner sep=1.7pt,fill] at (2,2) {};
    \node[draw,circle,inner sep=1.7pt,fill] at (3,3) {};
    
     \foreach \x in {1,...,6}{
     \coordinate[label=below:\x] (\x) at (\x,-0.2);
         \coordinate[label=left:\x] (\x') at (-0.2,\x);
         \coordinate (start\x) at (\x,-0.1);
         \coordinate ('start\x) at (-0.1,\x);
         \coordinate (end\x) at (\x,7);
         \coordinate ('end\x) at (7,\x);
     }
     
     \foreach \x in {4,...,6}{
     \foreach \y in {4,...,6}{
     \node[draw,circle,inner sep=1.7pt,fill] at (\x,\y) {};
     }}

     \foreach \x in {3,...,6}{
     \node[draw,circle,inner sep=1.7pt,fill=white] at (\x,2) {};
     \node[draw,circle,inner sep=1.7pt,fill=white] at (2,\x) {};
     }
   
\begin{scope}[on background layer]  
\foreach \x in {1,...,6}{
        \draw [thin, gray] (start\x) -- (end\x);
        \draw [thin, gray] ('start\x) -- ('end\x);}
\foreach \x in {1}{
        \draw [thin, gray] (start\x) -- (end\x);
        \draw [thin, gray] ('start\x) -- ('end\x);}
\end{scope}
\end{tikzpicture} \quad \quad
\begin{tikzpicture}[scale=0.65]
    \coordinate (Origin)   at (0,0);
    \coordinate[label=below:0] (zero) at (-0.6,-0.2);
    \coordinate (XMin) at (-1,0);
    \coordinate (XMax) at (8,0);
    \coordinate (YMin) at (0,-1);
    \coordinate (YMax) at (0,8);
    \draw [-latex] (XMin) -- (XMax);
    \draw [-latex] (YMin) -- (YMax);

    \node[draw,circle,inner sep=1.7pt,fill] at (0,0) {};
    \node[draw,circle,inner sep=1.7pt,fill] at (1,1) {};
    
     \foreach \x in {1,...,6}{
     \coordinate[label=below:\x] (\x) at (\x,-0.2);
         \coordinate[label=left:\x] (\x') at (-0.2,\x);
         \coordinate (start\x) at (\x,-0.1);
         \coordinate ('start\x) at (-0.1,\x);
         \coordinate (end\x) at (\x,7);
         \coordinate ('end\x) at (7,\x);
     }
     
     \foreach \x in {2,...,6}{
     \foreach \y in {2,...,6}{
     \node[draw,circle,inner sep=1.7pt,fill] at (\x,\y) {};
     }}
   
\begin{scope}[on background layer]  
\foreach \x in {1,...,6}{
        \draw [thin, gray] (start\x) -- (end\x);
        \draw [thin, gray] ('start\x) -- ('end\x);}
\foreach \x in {1}{
        \draw [thin, gray] (start\x) -- (end\x);
        \draw [thin, gray] ('start\x) -- ('end\x);}
\end{scope}
\end{tikzpicture}
\end{center}
\end{ex}

\subsection{The non-local case}

If $S=S_1 \times \dots \times S_r$ is a non-local good semigroup, using the facts that $J=M_1 \times \dots \times M_r$, $J-J=(M_1-M_1) \times \dots \times (M_r-M_r)$ and
$K(S)=K(S_1)\times \dots \times K(S_r)$, it is straightforward to generalize 
the definitions of almost symmetry (as it has been done in \cite{a-u}) and the previous theorems. 
More precisely, we will say that $S$ is \emph{almost symmetric} if its components
are almost symmetric for every $i=1, \dots, r$. Moreover, denoting by $ \boldsymbol{e_i}$
the multiplicity vector of $S_i$, we can state the following results. Their proofs 
are straightforward since, if $E,F$ are relative ideals of $S$ of the form 
$E_1\times \dots \times E_r$ and $F_1\times \dots \times F_r$, where the supports of
$E_i$ and $F_i$ are the same as the supports of $S_i$, then 
$E-F=(E_1-F_1) \times \dots \times (E_r-F_r)$ and hence we can reduce our proofs to the local case.

\begin{prop}\label{miracolo2} Preserving the above notations,
let $S_1 \times \dots \times S_r$ be a good semigroup and let $\boldsymbol{e}=(\boldsymbol{e_1} , \dots , \boldsymbol{e_r})$ and assume that $J-J$ is a good semigroup. Then the standard canonical ideal of $J-J$ is $$K(J-J)=(K(S)-(J-J))-\boldsymbol{e}.$$
\end{prop}

\begin{thm}\label{teorema centrale 3} A good semigroup $S$ is almost symmetric if and only if $J-J$ 
is a good relative ideal of $S$ and $J- \boldsymbol{e}$ 
is the standard canonical ideal of $J-J$.
\end{thm}

\begin{cor} A good semigroup $S$ is almost symmetric and  $J-J=J-\boldsymbol{e}$ if and only if $J-J$ is a symmetric good semigroup.
\end{cor}

\section{Almost Gorenstein rings}

We conclude this paper by applying the previous results to noetherian,
analytically unramified, residually rational, one-dimensional, reduced, semilocal rings.

As it is described in \cite{a-u}, under these hypotheses, the integral closure $\overline{R}$ of $R$ is a direct product
of discrete valuation rings $V_1 \times \dots \times V_h$, and therefore it is possible to associate with such a ring $R$
a value semigroup $v(R)$ that is a good subsemigroup of $\N^h$.
Note that the residually rational hypothesis means that all the residue fields of $R$
and $\overline{R}$ are isomorphic; if moreover the cardinality of these residue fields 
is bigger than $h$, then the connection between the ring and its value semigroup 
is particularly meaningful and we will assume these two hypotheses till the end of the paper. One of the reasons for this connection is
the fact that, if $I \subseteq J$ are two fractional ideals of $R$, then
the length $\ell(J/I)$ of $J/I$, as an $R$-module, can be computed looking at the value ideals $v(I)$ and $v(J)$ of $v(R)$; in particular, we will use the
fact that $I=J$ if and only if $v(I)=v(J)$ (see \cite[Corollary 2.5]{danna1}).
We also recall that algebroid curves belong to this class of rings
and that, in this case, the integer $h$ represents the number of branches of the singularity.

\medskip
If $R$ is semilocal with maximal ideals $\mathfrak m_1, \dots, \mathfrak m_j$
(with $j \leq h$), then $R \cong R_{\mathfrak m_1} \times \dots \times R_{\mathfrak m_j}$
(see \cite[Corollary 3.2]{a-u}). From this, it follows immediately that $v(R)=v(R_{\mathfrak m_1})\times \dots \times v(R_{\mathfrak m_j})$ is the unique representation of $v(R)$
as a direct product of local good semigroups. Moreover, the Jacobson radical
$J(R)$ coincides with $\mathfrak m_1R_{\mathfrak m_1}\times \dots \times \mathfrak m_jR_{\mathfrak m_j}$ and hence $v(J(R))=J$, where $J$ is the Jacobson ideal of $v(R)$.

\medskip
In this context, we give the following definition. 
\begin{defin}
A noetherian,
analytically unramified, one-dimensional, reduced local ring is said
to be \emph{almost Gorenstein} if $\mathfrak m \omega_R = \mathfrak m$,
where $\omega_R$ is a canonical ideal of $R$ with the property that 
$R\subseteq \omega_R \subseteq \overline{R}$.
 
If $R$ is semilocal,
following again \cite{a-u}, we will say that $R$ is almost Gorenstein if 
$R_{\mathfrak m_i}$ is almost Gorenstein for any maximal ideal $\mathfrak m_i$
of $R$. 
\end{defin}
In the local case, $(R, \mathfrak m)$ is almost Gorenstein if and only if $v(R)$ 
is almost symmetric and the Cohen-Macaulay type of $R$ equals the type of $v(R)$ 
(\cite[Proposition 3.7]{a-u}). In fact, setting $v(R)=S$ and $M=v(\mathfrak m)$, in that paper it is proved that if $R$ is almost Gorenstein then $v(\mathfrak m: \mathfrak m)=M-M=K(S) \cup \Delta (\boldsymbol{\gamma})$ and, therefore, we could restate that result saying that $(R,\mathfrak m)$
is almost Gorenstein if and only if $v(R)$ is almost symmetric
and $v(\mathfrak m: \mathfrak m)=M-M$
(notice that, in general, only the inclusion  $v(\mathfrak m: \mathfrak m)\subseteq M-M$ holds).
In light of Theorem \ref{main} we can further reformulate this fact in the following way: 

\begin{cor}\label{bdf} Preserving the above assumptions, set $S=v(R)$ and 
$M=v(\mathfrak m)$. Then  $(R,\mathfrak m)$
is almost Gorenstein if and only if $M-\boldsymbol{e}$ is the standard canonical 
ideal of $M-M$ and $v(\mathfrak m: \mathfrak m)=M-M$.
\end{cor}

This result is not particularly useful from a computational point of view, since it is not easy at all to determine $\mathfrak m: \mathfrak m$ and its value semigroup. However, if we restrict ourselves to the 
maximal embedding dimension case, we get a stronger result.

\begin{prop}\label{bella} Preserving the above assumptions, set $S=v(R)$, 
$M=v(\mathfrak m)$ and assume that $R$ is of maximal embedding dimension.
 The following conditions are equivalent:
\begin{itemize}
    \item[(i)] $(R,\mathfrak m)$ is almost Gorenstein;
    \item[(ii)] $M-\boldsymbol{e}$ is a good symmetric semigroup;
    \item[(iii)] $S$ is almost symmetric.
\end{itemize}
\end{prop}

\begin{proof}
$ $
\begin{description}
\item[$(i) \Leftrightarrow (ii)$] 
It is well known that $R$ is of maximal embedding dimension if and only if
$\mathfrak m:\mathfrak m = \mathfrak m x^{-1}$, where $x \in \mathfrak m$ is a minimal reduction of $\mathfrak m$, i.e., in 
our context, an element of minimal non-zero value (that is $\boldsymbol e$). Hence $v(\mathfrak m x^{-1})=M-\boldsymbol e$. Moreover, the equality $\mathfrak m:\mathfrak m = \mathfrak m x^{-1}$ holds if and only if $\mathfrak m x^{-1}$ is a ring.

Applying \cite[Proposition 25]{BF}, we get that $R$ is almost Gorenstein if and only if 
$\mathfrak m:\mathfrak m = \mathfrak m x^{-1}$ is a Gorenstein ring, that implies immediately that 
$M-\boldsymbol e=v(\mathfrak m x^{-1})$ is a good symmetric semigroup.

Conversely, if $M-\boldsymbol e$ is a good symmetric semigroup, since by hypothesis 
$\mathfrak m x^{-1}=\mathfrak m:\mathfrak m$ is a ring, it has to be Gorenstein, having a symmetric value semigroup. The thesis follows again applying \cite[Proposition 25]{BF}.

\item[$(ii) \Leftrightarrow (iii)$]  We need only to observe that $v(\mathfrak m:\mathfrak m)
\subseteq M-M \subseteq M-\boldsymbol e =v(\mathfrak m x^{-1})$ and so
the hypothesis that $R$ is of maximal embedding dimension implies 
that $\subseteq M-M \subseteq M-\boldsymbol e$, i.e. $S$ is of maximal embedding dimension. 
Now the thesis follows by Corollary \ref{alm-symm-good}
\end{description}
\end{proof}

\begin{ex}
The semigroup considered in Example \ref{2.7} is the value semigroup of 
the ring $R=k[[x,y,z,w]]/P_1\cap P_2$, where $k$ is any field of with at least $3$ elements, $P_1=(x^3-y^2, z-x^2, w-xy)$ and $P_2=(x^3-y^2,z,w)$. This fact can
be seen by considering the homomorphism $k[[x,y,z,w]]\rightarrow k[[t]]\times k[[u]]$
induced by $x \mapsto (t^2, u^2)$, $y\mapsto (t^3,u^3)$, $z\mapsto (t^4,0)$ and $w \mapsto
(t^5,0)$, whose kernel is $P_1 \cap P_2$. The ring $R$ is of maximal embedding dimension, since also its multiplicity is $4$ (i.e. the sum of the components of the minimal non-zero value), and
as explained in Example \ref{2.7}, $M-\boldsymbol e$ is a symmetric good semigroup.
Hence $R$ is almost Gorenstein.
\end{ex}

In light of the above discussion on the semilocal case, we can easily generalize the previous results.

\begin{cor}\label{rings} Let $R$ be a one-dimensional, noetherian, analytically unramified, residually rational, reduced 
semilocal ring. Set $S=v(R)$ and let $J=v(J(R))$ be its Jacobson ideal. The following conditions are equivalent:
\begin{itemize}
    \item[(i)] $R$ is almost Gorenstein;
    \item[(ii)] $J(R)$ is a canonical ideal of $J(R):J(R)$;
    \item[(iii)] $S$ is  almost symmetric and $v(J(R):J(R))=J-J$;
    \item[(iv)] $J$ is a canonical ideal of $J-J$ and $v(J(R):J(R))=J-J$.
\end{itemize}
\end{cor}

\begin{proof} $ $
\begin{description}
\item[$(i) \Leftrightarrow (ii)$]  By definition $R$ is almost Gorenstein if and only if $R_{\mathfrak m_i}$  is almost Gorenstein for every $i=1, \dots, j$. By \cite[Proposition 3.2]{DS2} this is equivalent to saying that 
$\mathfrak m_i R_{\mathfrak m_i}$ is a canonical ideal of $({\mathfrak m_iR_{\mathfrak m_i}}:{\mathfrak m_iR_{\mathfrak m_i}})$. Since $J(R)=\mathfrak m_1R_{\mathfrak m_1}\times \dots \times \mathfrak m_jR_{\mathfrak m_j}$ and $J(R):J(R)= ({\mathfrak m_1R_{\mathfrak m_1}}:{\mathfrak m_1R_{\mathfrak m_1}})
\times \dots \times ({\mathfrak m_jR_{\mathfrak m_j}}:{\mathfrak m_jR_{\mathfrak m_j}})$,
the thesis follows immediately.
\item[$(iii) \Leftrightarrow (iv)$] In both conditions it is assumed that $v(J(R):J(R))=J-J$, hence 
$J-J$ is good and the equivalence follows applying Theorem \ref{teorema centrale 3}.
\item[$(i) \Leftrightarrow (iii)$] By \cite{a-u} we know that, for every $i=1, \dots, j$, 
the ring $R_{\mathfrak m_i}$ is almost Gorenstein if and only if $S_i$ 
is almost symmetric and $v(\mathfrak m_iR_{\mathfrak m_i}: \mathfrak m_iR_{\mathfrak m_i})=M_i-M_i$.
The thesis is now straightforward.
\end{description}
\end{proof}

Like in the local case, the previous result is not very useful from a computational point of view, but we can state  a semilocal version of
Proposition \ref{bella}.

\begin{cor}\label{final}
We preserve the above notations. Let $R$ be a one-dimensional, noetherian, analytically unramified, residually rational, reduced 
semilocal ring; identify $R$ with $R_{\mathfrak m_1} \times \dots \times R_{\mathfrak m_j}$ and set $x=(x_1, \dots, x_j)$, where $x_i$ is a minimal reduction of
$\mathfrak m_iR_{\mathfrak m_i}$ (for any $i=1,\dots ,j$); set $S=v(R)$ and let $J=v(J(R))$ be its Jacobson ideal. Assume that $R_{\mathfrak m_i}$
is of maximal embedding dimension for every $i=1, \dots , j$ (or, equivalently, that $J(R)x^{-1}=J(R):J(R)$).
    Then the following conditions are equivalent:
\begin{itemize}
    \item[(i)] $R$ is almost Gorenstein;
    \item[(ii)] $J-\boldsymbol e$ is a good symmetric semigroup
    (where $\boldsymbol e =(\boldsymbol{e_1} , \dots , \boldsymbol{e_r})$ and $\boldsymbol e_i$ is the multiplicity vector of $S_i=v(R_{\mathfrak m_i}$).
\end{itemize}
\end{cor}

\noindent {\bf Statements and Declarations.}
On behalf of all authors, the corresponding author states that there is no conflict of interest.

\end{document}